\documentclass[final]{siamltex}
\usepackage{amssymb}
\usepackage{amsmath,amssymb,amsfonts}
\usepackage{graphicx}
\usepackage{epsfig}

\title{Spectral Indicator Method for A Non-selfadjoint Steklov Eigenvalue Problem
\thanks{The research of J. Liu was supported in part
by Guangdong Natural Science Foundation of China (2016A030313074). The research
of J. Sun  was supported in part by NSF Grant DMS-1521555.}}
\author{J. Liu
\thanks{Department of Mathematical Sciences, Jinan University, Guangzhou, 130012, China ({\tt liujuan@jnu.edu.cn}).}
\and J. Sun 
\thanks{Department of Mathematical Sciences, Michigan Technological University, Houghton, MI 49931, U.S.A. ({\tt jiguangs@mtu.edu}).}
\and T. Turner 
\thanks{Department of Mathematics and Computer Science, University of Maryland Eastern Shore, Princess Anne, MD 21853, U.S.A. ({\tt tdturner@umes.edu}).}
}
\date{}
\begin{document}
\maketitle

\begin{abstract}
We propose an efficient numerical method for a non-selfadjoint Steklov eigenvalue problem. The Lagrange finite element is used for discretization. The convergence is proved using the spectral perturbation theory for compact operators. The non-sefadjointness of the problem leads to non-Hermitian matrix eigenvalue problem. Due to the existence of complex eigenvalues and lack of a priori spectral information, we propose a modified version of the recently developed spectral indicator method to compute (complex) eigenvalues in a given region on the complex plane. In particular, to reduce computational cost, the problem is transformed into a much smaller matrix eigenvalue problem involving the unknowns only on the boundary of the domain. Numerical examples are presented to validate the effectiveness of the proposed method.
\end{abstract}

\section{Introduction}
Steklov eigenvalue problems arise in mathematical physics with spectral parameters in the boundary conditions \cite{Kuznetsov2014NAMS}.
Applications of Steklov eigenvalues include surface waves, mechanical oscillators immersed in a viscous fluid, 
the vibration modes of a structure in contact with an incompressible fluid, etc 
\cite{Kuznetsov2014NAMS, CanavatiMinzoni1979JMAA, CaoEtal2013SIAMNA}. Recently, Steklov eigenvalues have been used 
in the inverse scattering theory to reconstruct the index of refraction of an inhomogeneous media \cite{Cakoni2016SIAMAM}. 
Note that most Steklov eigenvalue problems considered
in the literature are related to partial differential equations of second order. However,
Steklov eigenvalue problems of higher order were also studied, e.g., the fourth order Steklov eigenvalue problem \cite{AnBiLuo2016JIA}.

In contrast to the theoretical study of the Steklov eigenvalue problem, numerical methods, in particular, finite element methods have attracted some researchers rather recently
\cite{AndreevTodorov2004IMANA, BiLiYang2016ANM, BiLiYang2016NMPDE, MoraEtal2015MMMA, CaoEtal2013SIAMNA, ArmentanoPadra2008ANM, KumarPrashat2010CMM,DelloRussoEtal2011CMA,JiaLuoFuXie2014AMSin}. 
Various methods have been proposed, including the isoparametric finite element method \cite{AndreevTodorov2004IMANA}, 
the virtual element method \cite{MoraEtal2015MMMA}, non-conforming finite element methods \cite{DelloRussoEtal2011CMA, JiaLuoFuXie2014AMSin},
the spectral-Galerkin method \cite{AnBiLuo2016JIA}, adaptive methods \cite{BiLiYang2016ANM}, multilevel methods \cite{Xie2014IMANA}, etc. 
All of the above works consider the selfadjoint cases.
In this paper, we consider a non-selfadjoint Steklov eigenvalue problem arising in the study of non-homogeneous absorbing medium in inverse scattering theory \cite{Cakoni2016SIAMAM}.
There seems to exist only one paper by Bramble and Osborn \cite{BrambleOsborn1972}, which considered the non-selfadjoint case.
However, the second order non-selfadjoint operator is assumed to be uniformly elliptic and no numerical results were reported in \cite{BrambleOsborn1972}.
In this sense, the current paper is the first paper contains both finite element theory and numerical examples for a non-selfadjoint Steklov eigenvalue problem, to the authors' knowledge.
For the general theory and examples of finite element methods for 
eigenvalue problems, we refer the readers to the book chapter by Babu\v{s}ka and Osborn \cite{BabuskaOsborn1991}, the review paper by Boffi \cite{Boffi2010AN}, 
and the recently published book by Sun and Zhou \cite{SunZhou2016}.

There are two major challenges to develop effective finite element methods for non-selfadjoint eigenvalue problems \cite{Osborn1975MC, BabuskaOsborn1991, SunZhou2016}. 
The first one is the lack of systematic tools to prove the convergence of the finite element discretization. In general, for an eigenvalue problem, the convergence of the finite
element method for the associated source problem needs to be established first, which is not as easy as the selfadjoint cases.
The second one is the lack of effective eigensolvers to compute the complex eigenvalues when no a priori spectral information is available. 
Finite element discretization of non-selfadjoint eigenvalue problems usually 
leads to non-Hermitian generalized matrix eigenvalue problems, which are very challenging in numerical linear algebra \cite{Saad2011}.

In a recent paper \cite{Huang2016JCP}, a novel spectral indicator eigenvalue solver {\bf RIM} (recursive integral method) is developed for non-Hermitian eigenvalue problems.
\textbf{RIM} computes all eigenvalues in a region on the complex plane $\mathbb C$ without any a priori spectral information. Roughly speaking, 
given a region $S \subset \mathbb C$ whose boundary $\Gamma:=\partial S$ is a simple closed curve, 
{\bf RIM} computes an indicator $\delta_S$ for $S$ using spectral projection defined by a Cauchy contour integral on $\Gamma$.
The indicator is used to decide if $S$ contains eigenvalue(s). In case of positive answers, $S$ is divided into sub-regions
and indicators for these sub-regions are computed. The procedure continues until the size of the region is smaller than 
a specified precision $d_0$ (e.g., $d_0 = 10^{-9}$).
The centers of the regions are the approximations of eigenvalues.
It is noted that contour integral is a classical tool in 
operator theory \cite{Kato1966}. It became popular recently to approximate eigenvalue problems using invariant subspaces \cite{SakuraiSugiura2003CAM, 
Polizzi2009PRB, Beyn2012LAA}.

In this paper, we propose a simple finite element method and expand the spectral indicator method {\bf RIM} to compute complex Steklov eigenvalues in the region of interest.
The contributions of the paper include 1) it provides a finite element analysis for a non-selfadjoint Steklov eigenvalue problem; 2) it reduces the computation
to the boundary of the domain, i.e., treating a much smaller discrete problem; and 3) it extends a new eigensolver for the resulting non-Hermitian matrix eigenvalue problem.

The rest of the paper is organized as follows. In Section 2, we introduce the Steklov eigenvalue problem, its adjoint problems, variational formulations, and prove the
well-posedness. In Section 3, we propose a linear finite element method and prove the convergence. In Section 4, we extend the new spectral indicator method 
for the resulting non-Hermitian matrix eigenvalue problems. 
In particular, to reduce computational cost, the problem is transformed into
a much smaller matrix eigenvalue problem involving the unknowns only on the boundary of the domain. 
Numerical examples are presented in Section 5. Finally, some conclusions and future works are discussed in Section 6.

\section{A Non-selfadjoint Steklov Eigenvalue Problem}
Let $\Omega \subset \mathbb R^2$ be a bounded polygon with Lipshitz boundary $\partial \Omega$. Let $\nu$ be the unit outward normal to $\partial \Omega$.
Let $k$ be the wavenumber and $n(x)$ be the index of refraction.
We consider the Steklov eigenvalue problem to find $\lambda \in \mathbb C$ and a nontrivial function $u \in H^1(\Omega)$ such that
\begin{subequations}\label{SteklovEig}
\begin{align}
\label{DirichletE} \triangle u + k^2n(x) u &=0 \qquad \text{in } \Omega,\\[1mm]
\label{DirichletB}\frac{\partial u}{ \partial \nu}+\lambda u &=0\qquad \text{on } \partial \Omega.
\end{align}
\end{subequations}
Define
\[
(u, v) = \int_\Omega u \overline{v} \,dx,  \qquad \langle f, g \rangle = \int_{\partial \Omega} f \overline g \,ds,
\]
and the continuous sesquilinear form
\[
a(u, v):= (\nabla u, \nabla v) - k^2(nu, v) \quad \text{for all } u, v \in H^1(\Omega).
\]

The weak formulation for \eqref{SteklovEig} is to find $(\lambda, u) \in \mathbb C \times H^1(\Omega)$ such that
\begin{equation}\label{SteklovEigWeak}
 a(u, v) = -\lambda\langle u, v \rangle \quad \text{for all } v \in H^1(\Omega).
\end{equation}
The associated source problem is, given $g \in L^2(\partial \Omega)$, to find $u \in H^1(\Omega)$ such that 
\begin{equation}\label{SteklovSWeak}
a(u, v) = \langle g, v \rangle \quad \text{for all } v \in H^1(\Omega).
\end{equation}
In this paper, we assume that $n(x)$ is a bounded complex valued function given by
\[
n(x) = n_1(x) + i \frac{n_2(x)}{k},
\]
where $i = \sqrt{-1}$, $n_1(x) > 0$ and $n_2(x) \ge 0$ are bounded smooth functions.

Define an operator $\mathcal{C}: H^1(\Omega) \to H^1(\Omega)$ which maps $u \in H^1(\Omega)$ to $w \in H^1(\Omega)$ satisfying
\[
(w, v)_{H^1(\Omega)}= k^2(n(x)u, v) \quad \text{for all } v \in H^1(\Omega).
\]
It is clear that the above problem has a unique solution. The regularity result for elliptic problems implies that $w \in H^2(\Omega)$ if $u \in H^1(\Omega)$ and $n(x) \in H^1(\Omega)$.
Hence $\mathcal{C}: u \to w$ is an compact operator \cite{BrennerScott2008}.  

It is easy to verify that $a(\cdot, \cdot)$ satisfies the G{\aa}rding's inequality \cite{BrennerScott2008}, i.e., there exist constants $K < \infty$ and $\alpha_0 > 0$ such that
\begin{equation}\label{Garding}
\text{Re}\left\{a(v, v)\right\} + K\|v\|^2_{L^2(\Omega)} \ge \alpha_0 \|v\|^2_{H^1(\Omega)} \quad \text{for all } u \in H^1(\Omega).
\end{equation}
Let $K$ be a positive constant, which is large enough. Define the sesquilinear form $A: H^1 \times H^1 \to \mathbb C$ such that
\begin{equation}\label{sesA}
A(u,v) := a(u, v) + K(u, v) = (\nabla u, \nabla v) - k^2(nu, v) + K(u, v), \quad u, v \in H^1(\Omega).
\end{equation}
The following lemma shows that $A$ is $H^1(\Omega)$-elliptic \cite{HsiaoWendland2008}.

\begin{lemma}\label{AisHelliptic}
For $K$ large enough, the sesquilinear form $A$ is $H^1(\Omega)$-elliptic, i.e., there exists $\alpha_0 > 0$ such that
\[
|A(v,v)| \ge \alpha_0 \|v\|^2_{H^1(\Omega)} \quad \text{for all } v \in H^1(\Omega).
\]
\end{lemma}
\begin{proof} Since $n_1(x)$ is bounded, there exist a constant $B$ such that $n_1(x) < B$ for all $x \in \Omega$.
Using the G{\aa}rding's inequality \eqref{Garding}, we have that
\begin{eqnarray*}
|A(v,v)|&=&  |(\nabla v, \nabla v) - k^2(nv, v) + K(v, v)| \\
&\ge& \text{Re} \left\{ (\nabla v, \nabla v) - k^2(nv, v) + K(v, v) \right\} \\
&=&  (\nabla v, \nabla v) - k^2(n_1(x)v, v) + K(v,v) \\
&\ge& (\nabla v, \nabla v) - k^2B(v, v) + K(v,v)\\
&\ge& \alpha_0 \|v\|^2_{H^1(\Omega)},
\end{eqnarray*}
where $\alpha_0=\min\{1, K-k^2B\}$ for $K$ large enough.
\end{proof}

As a consequence, the Fredholm alternative can be used to show the existence of a unique solution for \eqref{SteklovSWeak}. 
To this end, we need to define the (generalized) Neumann eigenvalues. 
\begin{definition}
The Neumann eigenvalue problem associated with $n(x)$ is to find $k^2 \in \mathbb C$ and a nontrivial $u \in H^1(\Omega)$ such that
\begin{subequations}\label{NeumannEig}
\begin{align}
\label{NeumannEE} \triangle u +k^2 n(x) u &=0\qquad \text{in } \Omega,\\[1mm]
\label{NeumannEB}\frac{\partial u}{ \partial \nu}&=0 \qquad \text{on } \partial \Omega.
\end{align}
\end{subequations}
\end{definition}
\begin{theorem}
Let $g \in L^2(\Omega)$. Assuming that $k^2$ is not a Neumann eigenvalue associated with $n(x)$ on $\Omega$, there exists a unique solution $u \in H^1(\Omega)$ to \eqref{SteklovSWeak} 
such that 
\begin{equation}\label{regularityu}
\|u\|_{H^1(\Omega)} \le C \|g\|_{L^2(\partial \Omega)}.
\end{equation}
\end{theorem}
\begin{proof}
Since $k^2$ is not a Neumann eigenvalue, then uniqueness holds for \eqref{SteklovSWeak}. By Fredholm Alternative (see e.g., Section 5.3 of \cite{HsiaoWendland2008}), 
there exists a unique solution $u$ to \eqref{SteklovSWeak}
and the regularity follows readily.
\end{proof}

Consequently, one can define an operator, which is in fact the Neumann-to-Dirichlet mapping, $T: L^2(\partial \Omega) \to L^2(\partial \Omega)$ \cite{Cakoni2016SIAMAM}
\begin{equation}\label{NtDmapping}
Tg = u|_{\partial \Omega}.
\end{equation}
The mapping $T$ is compact since $Tg \in H^{1/2}(\partial \Omega)$ and $H^{1/2}(\partial \Omega)$ is compactly embedded in $L^2(\partial \Omega)$.
Denote an eigenpair of $T$ by  $(\mu, g)$ such that
\[
Tg = \mu g. 
\]
It is clear that $\mu$ and $\lambda$ are related
\[
 \lambda = -1/\mu.
\]

We shall also need the adjoint operator $T^*$ of $T$ for the proof of convergence of the finite element method later. 
Consider the adjoint problem for \eqref{SteklovSWeak}. Given $g \in L^2(\partial \Omega)$, find $v \in H^1(\Omega)$ such that
\begin{equation}\label{adjoint}
a(u, v) = \langle u, g\rangle \quad \text{for all } u \in H^1(\Omega).
\end{equation}
Then \eqref{adjoint} has a unique solution $v$. The solution operator for \eqref{adjoint} is the adjoint operator $T^*g: L^2(\partial \Omega) \to L^2(\partial \Omega)$ such that $T^*g = v|_{\partial \Omega}$.

\section{Finite Element Approximation}
In this section, we present a finite element approximation $T_h$ for $T$. 
Let $\mathcal{T}_h$ be a regular triangular mesh for $\Omega$ with mesh size $h$.
Let $V_h \subset H^1(\Omega)$ be the Lagrange finite element space associated with $\mathcal{T}_h$
and $V^B_h := V_h|_{\partial \Omega}$ be the restriction of $V_h$ on $\partial \Omega$. It is clear that $V^B_h \subset L^2(\partial \Omega)$.
The finite element formulation for the Steklov eigenvalue problem is to find $(\lambda_h, u_h) \in \mathbb C \times V_h$ such that
\begin{equation}\label{SteklovEDWeak}
(\nabla u_h, \nabla v_h) - k^2(nu_h, v_h) = -\lambda_h \langle u_h, v_h \rangle \quad \text{for all } v_h \in V_h.
\end{equation}

For the convergence of eigenvalues, we first study the source problem.
Given $g \in L^2(\partial \Omega)$, let $g_h$ be the projection of $g$ onto $V^B_h$. The discrete problem is to find $u_h \in V_h$ such that 
\begin{equation}\label{SteklovSDWeak}
(\nabla u_h, \nabla v_h) - k^2(nu_h, v_h) = \langle g_h, v_h \rangle \quad \text{for all } v_h \in V_h.
\end{equation}
In the rest of this section, we assume that $u \in H^2(\Omega)$ and the same regularity holds for the solution of the adjoint problem. This is the case when $\Omega$ is convex.
In general, when $\Omega$ is non-convex, $u$ does not belong to $H^2(\Omega)$.  
We refer the readers to \cite{Grisvard1985} for further discussions on the regularity of $u$, which is out the scope of the current paper. 

We have the following estimate for \eqref{SteklovSDWeak} (see Theorem 5.7.6 of \cite{BrennerScott2008}).
\begin{theorem}\label{ErrEstSource}
Let $u$ be the solution to \eqref{SteklovSWeak}. Assume that $k$ is not a Neumann eigenvalue. 
There exists a unique solution $u_h$ to \eqref{SteklovSDWeak} such that, for $h$ small enough,
\begin{equation}\label{uuhinfH1}
\|u-u_h\|_{H^1(\Omega)} \le C \inf_{v \in V_h} \|u-v\|_{H^1(\Omega)}
\end{equation}
and
\begin{equation}\label{uuhL2h}
\|u-u_h\|_{L^2(\Omega)} \le C h \|u-u_h\|_{H^1(\Omega)}.
\end{equation}
\end{theorem}

\begin{proof}
For the finite dimensional problem \eqref{SteklovSDWeak}, existence of a solution can be established using uniqueness.
Assuming that there exists a nontrivial solution $u_h$ to \eqref{SteklovSDWeak} for $g_h=0$. For the continuous problem,
$g=0$ implies that the solution $u=0$. Then \eqref{ErrH1} asserts that $u_h=0$.  Thus the uniqueness holds,
which implies the existence of the solution $u_h$ as well.

Using \eqref{SteklovSWeak} and \eqref{SteklovSDWeak}, one has the Galerkin orthogonality
\[
a(u-u_h, v_h) = 0 \quad \text{for all } v_h \in V_h.
\]
The G{\aa}rding's inequality \eqref{Garding} implies that
\begin{eqnarray*}
\alpha_0 \|u-u_h\|^2_{H^1(\Omega)} & \le & |a(u-u_h, u-u_h) + K(u-u_h, u-u_h)| \\
&=&| a(u-u_h, u-v_h) + K\|u-u_h\|^2_{L^2(\Omega)} | \\
&\le& C \|u-u_h\|_{H^1(\Omega)} \|u-v_h\|_{H^1(\Omega)}  + K\|u-u_h\|^2_{L^2(\Omega)}.
\end{eqnarray*}
Assume that the estimate \eqref{uuhL2h} holds, i.e., 
\begin{equation}\label{L2errest}
\|u-u_h\|_{L^2(\Omega)} \le C_1h \|u-u_h\|_{H^1(\Omega)}
\end{equation}
for some constant $C_1 > 0$. One has that 
\begin{equation}\label{ErrH1}
\alpha_0 \|u-u_h\|^2_{H^1(\Omega)}  \le C \|u-u_h\|_{H^1(\Omega)} \|u-v_h\|_{H^1(\Omega)}  + KC_1h^2 \|u-u_h\|^2_{H^1(\Omega)}.
\end{equation}
Then for $h$ small enough, we obtain
\[
\|u-u_h\|_{H^1(\Omega)} \le C \inf_{v \in V_h} \|u-v_h\|_{H^1(\Omega)} \quad \text{for all } v_h \in V_h.
\]

The rest of the proof is devoted to verify \eqref{uuhL2h}.
Let $w$ be the solution to the adjoint problem
\[
a(v, w) = (u-u_h, v) \quad \text{for all } v \in V. 
\]
Then, for any $w_h \in V_h$,
\begin{eqnarray*}
(u-u_h, u-u_h)  &=& a(u-u_h, w) \\
		&=& a(u-u_h, w-w_h) \\
		&\le& C\|u-u_h\|_{H^1(\Omega)} \|w-w_h\|_{H^1(\Omega)}\\
		&\le& Ch \|u-u_h\|_{H^1(\Omega)} |w|_{H^2(\Omega)} \\
		&\le& Ch \|u-u_h\|_{H^1(\Omega)} \|u-u_h\|_{L^2(\Omega)},
\end{eqnarray*}
where we have used the regularity of the solution for the adjoint problem. Consequently,
\[
\|u-u_h\|_{L^2(\Omega)} \le C h \|u-u_h\|_{H^1(\Omega)}.
\]
\end{proof}

As a result, problem \eqref{SteklovSDWeak} defines a discrete operator $T_h: L^2(\partial \Omega) \to V^B_h$ such that 
\begin{equation}\label{Thguh}
 T_h g =  u_h|_{V^B_h}.
 \end{equation}
The following theorem shows that $T_h$ converges to $T$ in norm in $L^2(\partial \Omega)$. 
\begin{theorem}\label{TThL2} Assume that $g \in H^{1/2}(\partial \Omega) \subset L^2(\partial \Omega)$. 
Let $T$ and $T_h$ be defined as in \eqref{NtDmapping} and \eqref{Thguh} using linear Lagrange element, respectively. Then
\begin{equation}\label{TThL2E}
\|T-T_h\|_{L^2(\partial \Omega), L^2(\partial \Omega)} \le C h^{3/2}.
\end{equation}
\end{theorem}
\begin{proof}
Using the approximation property of the linear Lagrange finite element (see Eqn. (3.9) of \cite{SunZhou2016}), for $u \in H^2(\Omega)$, one has that
\begin{equation}\label{infvVh}
 \inf_{v \in V_h} \|u-v\|_{H^1(\Omega)} \le Ch\|u\|_{H^2(\Omega)} \le C h\|g\|_{L^2(\partial \Omega)}.
\end{equation}
Therefore, by Theorem \ref{ErrEstSource}, 
\[
\|u-u_h\|_{L^2(\Omega)} \le C h^2 \|g\|_{L^2(\partial \Omega)}.
\]
One has that
\begin{eqnarray*}
\|(T-T_h)g\|_{L^2(\partial \Omega)} &=& \|(u-u_h)|_{\partial \Omega}\|_{L^2(\partial \Omega)} \\
	&\le& C \|u-u_h\|_{L^2(\Omega)}^{1/2}  \|u-u_h\|_{H^1(\Omega)}^{1/2} \\
	&\le& Ch^{3/2}\|g\|_{L^2(\partial \Omega)},
\end{eqnarray*}
where we have applied Theorem~\ref{ErrEstSource} and the trace theorem (Theorem 1.6.6 in \cite{BrennerScott2008}). Hence \eqref{TThL2E} follows immediately and the proof is complete.
\end{proof}

Similarly, the discrete problem for the adjoint problem \eqref{adjoint} can be defined as follows. Given $g \in L^2(\partial \Omega)$, find $v_h \in H^1(\Omega)$ such that
\begin{equation}\label{SteklovSADWeak}
(\nabla u_h, \nabla v_h) - k^2(nu_h, v_h) = \langle u_h, g \rangle \quad \text{for all } u_h \in V_h.
\end{equation}
Then all results in this section also hold for the adjoint problem \eqref{adjoint}. In particular, one has that discrete adjoint operator  $T^*_h: L^2(\partial \Omega) \to V^B_h$ such that $ T_h g_h = v_h|_{V^B_h} $ 
where $v_h$ the solution of \eqref{SteklovSADWeak} and
\begin{equation}\label{TstartTh}
\|T^*-T^*_h\|_{L^2(\partial \Omega), L^2(\partial \Omega)} \le C h^{3/2}.
\end{equation}

The rest of the section is devoted to the finite element spectral approximation of Steklov eigenvalues based on the theory in \cite{Osborn1975MC}.
We shall need some preliminaries on the spectral theory of compact operators (see, e.g., \cite{Kato1966}). 
Let $T: X \to X$ be a compact operator on a complex Hilbert space $X$.
Let $z \in \mathbb C$. The resolvent operator of $T$ is defined as 
\begin{equation}\label{resolventT}
R_z(T) = (z-T)^{-1}.
\end{equation}
The resolvent set of $T$ is 
\begin{equation}\label{rhoT}
\rho(T)=\{ z \in \mathbb C: (z-T)^{-1} \text{ exists and is bounded} \}.
\end{equation}
The spectrum of $T$ is $\sigma(T)=\mathbb C \setminus \rho(T)$.

Since $T$ is compact, each $\mu \in \sigma(T)$ is an isolated eigenvalue of $T$ and the generalized
eigenspace associated with $\mu$ is finite dimensional.
Furthermore, there exists a smallest positive integer $\alpha$ such that
\[
\mathcal{N}\left((\mu-T)^\alpha\right) = \mathcal{N}\left((\mu-T)^{\alpha+1}\right),
\]
where $\mathcal{N}$ denotes the null space.
The integer $m=\dim  \mathcal{N}\left((\mu-T)^\alpha \right)$ is called the algebraic multiplicity of $\mu$. 
The functions in  $\mathcal{N}\left((\mu-T)^\alpha \right)$ are called the generalized eigenfunctions
of $T$ corresponding to $\mu$. Note that the geometric multiplicity of $\mu$ is defined as $\dim \mathcal{N}(\mu - T)$.

Let $\Gamma$ be a simple closed curve on the complex plane $\mathbb C$ lying in $\rho(T)$, which contains an eigenvalue $\mu$ and no other eigenvalues.
Let the algebraic multiplicity of $\mu$ be $m$.
The spectral projection is defined by
\[
E(\mu):= \frac{1}{2\pi i} \int_\Gamma R_z(T) dz.
\]
It is well-known that $E$ is a projection onto  the space spanned by the generalized eigenfunctions ${\phi}_j, j=1,\ldots, m$ associated with $\mu$, i.e.,
the range of $E$, $\mathcal{R}(E)$, coincides wth $ \mathcal{N}\left((\mu-T)^\alpha \right)$. 

Since $T_h$ converges to $T$ in norm as $h \to 0$, $\Gamma \subset \rho(T_h)$ for $h$ small enough. 
In addition, there exists $m$ eigenvalues $\mu_h^1, \ldots, \mu_h^m$ of $T_h$ inside $\Gamma$ such that
\[
\lim_{h \to 0} \mu^j_h = \mu_j \quad \text{for } j = 1, \ldots, m.
\]
The spectral projection
\[
E_h(\mu):= \frac{1}{2\pi i} \int_\Gamma R_z(T_h) dz
\] 
converges to $E$ pointwise and $\text{dim}\mathcal{R}(E_h) =  \text{dim}\mathcal{R}(E)$. 

If $\mu$ is an eigenvalue of $T$, then $\overline{\mu}$ is an eigenvalue of $T^*$. Let $\phi_1, \ldots, \phi_m$ be a basis for ${\mathcal R}(E)$ and
$\phi_1^*, \ldots, \phi_m^*$ be the dual basis to $\phi_1, \ldots, \phi_m$ (see Section 1.1 of \cite{SunZhou2016}). The following lemma (Theorem 3 of \cite{Osborn1975MC})
will be used to prove the convergence of Steklov eigenvalues.
\begin{lemma}
Let $\mu$ be an eigenvalue of $T$ with algebraic multiplicity $m$. Let $\mu_h^1, \ldots, \mu_h^m$ be the $m$ eigenvalues of $T_h$ converge to $\mu$ and define
\[
\hat{\mu}_h = \frac{1}{m} \sum_{j=1}^m \mu_h^j.
\]
Then there exists a constant $C$ such that
\begin{equation}\label{mumuh}
|\mu - \hat{\mu}_h| \le \frac{1}{m} \sum_{j=1}^m |\langle (T-T_h) \phi_j, \phi_j^* \rangle| + C \|(T-T_h)|_{{\mathcal R}(E)}\| \, \|(T^*-T^*_h)|_{{\mathcal R}(E^*)}\|,
\end{equation}
where ${\mathcal R}(E^*) = \text{span}\{ \phi_1^*, \ldots, \phi_m^*\}$.
\end{lemma}

Using the convergence results of the finite element method for the source problem and the above lemma, we have the following theorem.
\begin{theorem}
Let $\mu$ be an eigenvalue of $T$ with multiplicity $m$ and $\mu_h^j, j=1,\ldots,m$ be the $m$ eigenvalues of $T_h$ approximating $\mu$. Then there exists a constant $C$, independent of $h$,
such that
\[
|\mu - \hat{\mu}_{h}| \le C h^2, \qquad \text{where } \hat{\mu}_{h} = \frac{1}{m} \sum_{j=1}^m \mu_h^j.
\]
\end{theorem}
\begin{proof} Let $u_j$ and $u_{h,j}$ be the solutions of \eqref{SteklovSWeak} and \eqref{SteklovSDWeak} with right hand side being $\phi_j$, respectively.
Let $u_j^*$ and $u_{h,j}^*$ be the solution of the adjoint problem \eqref{adjoint} and the corresponding finite element solution with right hand side being $\phi_j^*$, respectively.
In view of Theorem \ref{TThL2} and \eqref{TstartTh}, we only need to estimate the first term of \eqref{mumuh}.
\begin{eqnarray*}
|\langle (T-T_h) \phi_j, \phi_j^*\rangle| &=&  |\langle \phi_j, (T^*-T^*_h) \phi_j^* \rangle| \\	
				&=& | a(u_j,  u_j^*-u_{h,h}^*)| \\
				&=& | a( u_h -u_{h,j},  u_j^*-u_{h,j}^*)| \\
				&\le& C \| u_h -u_{h,j}\|_{H^1(\Omega)}  \|u_j^*-u_{h,j}^*)\|_{H^1(\Omega)}\\
				&\le& Ch^2.
\end{eqnarray*}
\end{proof}

As a consequence, we have the following convergence result on the Steklov eigenvalues.
\begin{theorem}
Let $\lambda$ be a Steklov eigenvalue with multiplicity $m$ and $\lambda_h^j, j=1,\ldots,m$ be the $m$ discrete eigenvalues of \eqref{SteklovEDWeak} approximating $\lambda$. 
Then there exists a constant $C$, independent of $h$,
such that
\[
|\lambda - \hat{\lambda}_{h}| \le C h^2, \qquad \text{where } \hat{\lambda}_{h} = \frac{1}{m} \sum_{j=1}^m \lambda_h^j.
\]
\end{theorem}

\section{Spectral Indicator Method} 
When $n(x)$ is complex, the Steklov eigenvalue problem is non-selfadjoint.
The above finite element method leads to a non-Hermitian matrix eigenvalue problem.
Due to the lack of a priori spectral information, classical methods do not work effectively. To this end, 
we extend the spectral indicator method {\bf RIM}, which was proposed recently in \cite{Huang2016JCP} (see also \cite{Huang2017}) for the 
non-selfadjoint transmission eigenvalue problem \cite{ColtonMonkSun2010IP, Sun2011SIAMNA},
to compute (complex) Steklov eigenvalues in a given region on the complex plane $\mathbb C$.

The matrix form for \eqref{SteklovEDWeak} is given by
\begin{equation}\label{MF}
(G-k^2 M_n) {\boldsymbol u} = -\lambda M_{\partial \Omega} {\boldsymbol u},
\end{equation}
where $G$ is the stiffness matrix, $M_n$ is the mass matrix, $M_{\partial \Omega}$ is the mass matrix on $\partial \Omega$.
The standard way is to approximate $\lambda's$ by solving the generalized matrix eigenvalue problem \eqref{MF}. 

We first introduce {\bf RIM} proposed in \cite{Huang2016JCP} for the generalized eigenvalue problem 
\begin{equation}\label{AxLambdaBx}
A{\boldsymbol x} = \lambda B{\boldsymbol x},
\end{equation}
where $A=(G-k^2 M_n)$ and $B = -M_{\partial \Omega}$.

Let $S \subset \mathbb C$ be a simply connect domain and $\Gamma=\partial S$. 
The problem of interest is to compute all eigenvalues of \eqref{AxLambdaBx} in $S$.
Let ${\boldsymbol g}$ be a random vector. From the previous section, the spectral projection of  ${\boldsymbol g}$ is defined as
\[
E {\boldsymbol g} =\frac{1}{2\pi i} \int_\Gamma R_z(A,B) {\boldsymbol g} dz=  \frac{1}{2\pi i} \int_\Gamma (A-zB)^{-1}{\boldsymbol g} dz,
\]

The idea behind {\bf RIM} is very simple. The spectral projection $E {\boldsymbol g}$ can be used to
to decide if there exist eigenvalues in $S$ or not. If there are no eigenvalues inside $\Gamma$, $|E {\boldsymbol g}| = 0$. 
Otherwise, if there exist $m$ eigenvalues $\lambda_j, j = 1, \ldots, m,$ $|E {\boldsymbol g}| \ne 0$.

Without loss of generality, let $S$ be a square. $E{\boldsymbol g}$ can be approximated using a quadrature
\begin{equation}\label{XLXf}
E{\boldsymbol g} \approx  \dfrac{1}{2 \pi i} \sum_{j=1}^W \omega_j {\boldsymbol x}_j,
\end{equation}
where $\omega_j$'s are quadrature weights and ${\boldsymbol x}_j$'s are the solutions of the linear systems
\begin{equation}\label{linearsys}
(A- z_jB){\boldsymbol x}_j = {\boldsymbol g}, \quad j = 1, \ldots, W.
\end{equation}
Recall that if there is no eigenvalue inside $\Gamma$, then 
$E{\boldsymbol g} = {\bf 0}$ for all 
${\boldsymbol g} \in \mathbb C^n$. Hence $|E{\boldsymbol g}|$ can be used as an indicator of $S$.
However, in practice, it is difficult to distinguish between $|E{\boldsymbol g}| \ne 0$ and $|E{\boldsymbol g}|= 0$. 
The solution in \cite{Huang2016JCP} is to normalize
$E{\boldsymbol g}$ and project it again. 
The indicator $\delta_S$ is set to be
\begin{equation}\label{sigmaP2P}
\delta_S := \left | E \left( \frac{E {\boldsymbol g}}{|E {\boldsymbol g}|}\right)\right|.
\end{equation}
Since we use numerical quadratures, $\delta_S \approx 1$ if there exist eigenvalues in $S$.
In this case, $S$ is divided into sub-regions
and indicators for these sub-regions are computed. The procedure continues until the size of the region is smaller than 
a specified precision $d_0$ (e.g., $d_0 = 10^{-9}$).
Then the centers of the regions are the approximations of eigenvalues.

According to \eqref{sigmaP2P}, $\delta_S = 1$ if there exists at least one eigenvalue in $S$ and $\delta_S = 0$ if there is no eigenvalue in $S$.
Since only "yes" ($\delta_S = 1$) or "no" ($\delta_S = 0$) is needed in the algorithm and quadrature is used to evaluate $E{\boldsymbol g}$, it is natural to use
a threshold $\delta_0$ to distiguish "yes" and "no". 
Let ${\boldsymbol g}$ be a random vector and $\delta_0, 0 < \delta_0 < 1$, be a threshold value.
The following is the basic algorithm for {\bf RIM}.
\vskip 0.2cm

\begin{itemize}
\item[] {\bf RIM}$(A, B, S, d_0, \delta_0, {\boldsymbol g})$
\item[]{\bf Input:}  matrices $A, B$, region $S$, precision $d_0$, threshold $\delta_0$, random vector ${\boldsymbol g}$.
\item[]{\bf Output:}  generalized eigenvalue(s) $\lambda$ inside $S$ 
\vskip 0.1cm
\item[1.] Compute ${\delta_S}$ using \eqref{sigmaP2P}, \eqref{XLXf} and \eqref{linearsys}.
\item[2.] Decide if $S$ contains eigenvalue(s).
	\begin{itemize}
		\item If $\delta_S < \delta_0$, then exit.
		\item Otherwise, compute the size $h(S)$ of $S$.
			\begin{itemize}
				\item[-] If $h(S)  > d_0 $, 
						\begin{itemize}
						\item[] partition $S$ into subregions $S_j, j=1, \ldots J$.
						\item[] for $j=1: J$
						\item[] $\qquad${\bf RIM}$(A, B, S_j, d_0, \delta_0, {\boldsymbol g})$.
						\item[] end
						\end{itemize}
				\item[-] If $h(S) \le d_0$, 
						\begin{itemize}
							\item[] set $\lambda$ to be the center of $S$.
							\item[] output $\lambda$ and exit.
						\end{itemize}
			\end{itemize}
	\end{itemize}
\end{itemize}
\vskip 0.2cm

The computational cost of {\bf RIM} mainly comes from solving the linear system \eqref{linearsys} at each quadrature point.
Note that these matrices are $N \times N$, where $N$ is the number of vertices of the triangular mesh if linear Lagrange element is used.
Furthermore, for robustness, the strategy of {\bf RIM} in \cite{Huang2016JCP} selects a small threshold $\delta_0=0.1$, i.e., $S$ contains eigenvalue(s) whenever $\delta_S > \delta_0$.
This choice of threshold for selecting a region systematically leans towards further investigation of regions that may potentially contain eigenvalues.
Such a strategy leads to more linear systems.

To employ {\bf RIM} in a more efficient way for Steklov eigenvalues,
we consider an alternative matrix eigenvalue problem of a much smaller size by using a matrix version of $T_h$ directly.
This is possible due to the fact that the eigenvalues appear in the boundary condition such that one can rewrite the system
as an eigenvalue problem involving degrees of freedom only related to the boundary of the domain. Consequently, the size of the problem is reduced significantly.

From the finite element approximation,  $T_h: V^B_h \to V^B_h$ has the following matrix form, which is also denoted by $T_h$, 
\[
T_h : = I_h (G-k^2 M_n)^{-1} M_{\partial \Omega},
\] 
where $I_h$ corresponds the restriction of a function in $V_h$ onto $V^B_h$. $T_h$ is an $M\times M$ matrix where $M$
is the number of vertices on $\partial \Omega$. Clearly, $M \ll N$ and one only needs to consider a much smaller eigenvalue problem than \eqref{MF}
\begin{equation}\label{Thmu}
T_h {\boldsymbol u} = \mu {\boldsymbol u}.
\end{equation}
In this case, for the spectral projection $E{\boldsymbol g}$, the linear systems are of size $M$
\begin{equation}\label{linearsysTh}
(T- z_jI){\boldsymbol x}_j = {\boldsymbol g}, \quad j = 1, \ldots, W.
\end{equation}

The following is the modified version of {\bf RIM} designed for the Steklov eigenvalue problem using \eqref{Thmu}, denoted by {\bf S-RIM}.
\vskip 0.2cm

\begin{itemize}
\item[] {\bf S-RIM}$(T_h, S, d_0, \delta_0, {\boldsymbol g})$
\item[]{\bf Input:}  matrix $T_h$, region $S$, precision $d_0$, threshold $\delta_0$, random vector ${\boldsymbol g}$.
\item[]{\bf Output:}  eigenvalue(s) $\mu$ inside $S$ 
\vskip 0.1cm
\item[1.] Compute ${\delta_S}$ using \eqref{sigmaP2P}, \eqref{XLXf}, \eqref{linearsysTh}.
\item[2.] Decide if $S$ contains eigenvalue(s).
	\begin{itemize}
		\item If $\delta_S < \delta_0$, then exit.
		\item Otherwise, compute the size $h(S)$ of $S$.
			\begin{itemize}
				\item[-] If $h(S)  > d_0 $, 
						\begin{itemize}
						\item[] partition $S$ into subregions $S_j, j=1, \ldots J$.
						\item[] for $j=1: J$
						\item[] $\qquad${\bf S-RIM}$(T_h, S_j, d_0, \delta_0, {\boldsymbol g})$.
						\item[] end
						\end{itemize}
				\item[-] If $h(S) \le d_0$, 
						\begin{itemize}
							\item[] set $\mu$ to be the center of $S$.
							\item[] output $\mu$ and exit.
						\end{itemize}
			\end{itemize}
	\end{itemize}
\end{itemize}
%
%
%
%
%
%
%
%

\section{Numerical Examples}
We present some numerical results in this section. For all examples, we choose $k=1$. Consider three domains: $\Omega_1$ is the unit disk, $\Omega_2$ is the square whose vertices are
\[
(0, -1), \quad (1, 0), \quad (0, 1), \quad (-1, 0),
\] 
and $\Omega_3$ is an L-shaped domain
given by 
\[
(-0.9, 1.1) \times (-1.1, 0.9) \setminus [0.1, 1.1] \times [-1.1, -0.1].
\] 

For the disk with radius $R$ and constant index of refraction $n$, separation of variables in polar coordinates can be used to obtain exact Steklov eigenvalues.
Since $u$ is the solution of the Helmholtz equation \eqref{DirichletE}, it has the expansion
\begin{equation}\label{expansion}
u(r,\theta)=\sum_{m=-\infty}^{+\infty}a_m J_{|m|}(k\sqrt{n}r)e^{im\theta},\quad r<R, \quad \theta\in (0,2\pi],
\end{equation} 
where $m$'s are integers and $J_{|m|}$ denotes the Bessel function of order $|m|$.
By the boundary condition \eqref{DirichletB}, the coefficients  $a_m$ satisfy
\[
\sum_{m=-\infty}^{+\infty} a_m\bigg (k\sqrt{n} J_{|m|}'(k\sqrt{n}R)+\lambda J_{|m|}(k\sqrt{n}R)\bigg)e^{im\theta}=0,
\]
i.e.,
\begin{equation}\label{coefficient}
a_m\bigg (k\sqrt{n} J_{|m|}'(k\sqrt{n}R)+\lambda J_{|m|}(k\sqrt{n}R)\bigg)=0, \ \ \ m=-\infty, \ldots,\infty.
\end{equation}
If $\lambda$ is a Steklov eigenvalue, there exists at least one $m$ such that $a_m \ne 0$. Then from \eqref{coefficient},  $\lambda$ must satisfy
\begin{equation*}
k\sqrt{n} J_{m}'(k\sqrt{n}R)+\lambda J_{m}(k\sqrt{n}R)=0, \ \ \textrm{for some}\ m \ge 0.
\end{equation*}
 Therefore, the Steklov eigenvalues are given by
\begin{equation*}
\lambda=-\frac{k\sqrt{n} J_{m}'(k\sqrt{n}R)}{J_{m}(k\sqrt{n}R)}, \ \ \textrm{for some}\ m \ge 0.
\end{equation*}
On the other hand, it is clear that all $\lambda_m=-\frac{k\sqrt{n} J_{m}'(k\sqrt{n}R)}{J_{m}(k\sqrt{n}R)},\ m=0, 1, 2, \cdots,$ are Steklov eigenvalues since 
$J_{m}(k\sqrt{n}r)e^{im\theta}$ are non-trivial solutions of the Steklov eigenvalue problem.

From the above discussion, for the unit disk, Stekloff eigenvalues are given by
\begin{equation}\label{DiskLambda}
\lambda_m = - k \sqrt{n} \frac{J'_m(k\sqrt{n})}{J_m(k\sqrt{n})}, \quad m = 0, 1, 2, \ldots.
\end{equation}
In Fig.~\ref{lambdam}, we show $\lambda_m$ against the index of refraction $n$ for $m=0, 1, 2$. 
\begin{figure}
\begin{center}
{ \scalebox{0.5} {\includegraphics{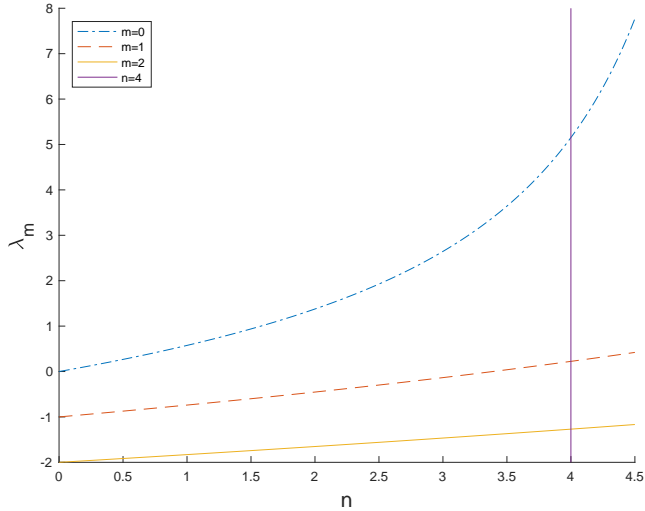}}}
\caption{$\lambda_m$ v.s. $n$ for $m=0,1,2$.}
 \label{lambdam}
\end{center}
\end{figure}
Using \eqref{DiskLambda}, when $n=4$ the $6$ largest Steklov eigenvalues are 
\begin{equation}\label{realnSteklov}
\lambda_1=5.151841, \quad \lambda_{2,3}=0.223578, \quad \lambda_{4,5}=-1.269100,  \quad \lambda_6=-2.472703
\end{equation}
and when $n=4+4i$ the $4$ complex Steklov eigenvalues with largest imaginary parts are
\begin{equation}\label{compnSteklov}
\lambda_1=\begin{array}{l} -0.320506 \\  \quad+ 3.124689i \end{array}, \quad  \lambda_{2,3}=\begin{array}{l} -0.136861 \\  \quad+ 1.396737i\end{array}, \quad
\lambda_4=\begin{array}{l} -1.353076 \\  \quad+ 0.791723i\end{array}.
\end{equation}


\subsection{Selfadjoint Cases}
When the index of refraction $n(x)$ is real, the problem is selfadjoint and all Steklov eigenvalues are real.
We compute the six largest Steklov eigenvalues for $n(x) = 4$ on a series of uniformly refined meshes for each domain.
The results are shown in Tables \ref{EigR1}, \ref{EigS1}, and \ref{EigL1}. The mesh sizes are denoted by $h$.
The values are consistent with those in \cite{Cakoni2016SIAMAM}.
\begin{center}
\begin{table}[h!]
\begin{center}
\begin{tabular}{lrrrrrr}
\hline
 $h$& 1st&2nd&3rd&4th&5th&6th\\
\hline \hline
0.2341& 5.016606&0.206380&0.205917&-1.294039&-1.294339&-2.561531\\
0.1208&5.116979&0.219175&0.219048&-1.275370&-1.275440&-2.494866\\
0.0613&5.143045&0.222469&0.222436&-1.270670&-1.270687&-2.478245\\
0.0309&5.149636&0.223301&0.223292&-1.269493&-1.269497&-2.474088\\
0.0155&5.151289&0.223509&0.223507&-1.269198&-1.269199&-2.473049\\
\hline
\end{tabular}
\end{center}
\caption{The largest six Steklov eigenvalues for the circle $n(x)=4$.}
\label{EigR1}
\end{table}
\end{center}

Note that for the unit disk, the first $3$ eigenvalues are given by \eqref{DiskLambda} for $m=0, 1, 2$. The values of the columns 2, 3, 4 in Table~\ref{EigR1} approximate the intersections
of $\lambda_m, m=0, 1, 2$ and $n=4$ in Fig.~\ref{lambdam}. 

\begin{center}
\begin{table}[h!]
\begin{center}
\begin{tabular}{lrrrrrr}
\hline
 $h$& 1st&2nd&3rd&4th&5th&6th\\
\hline \hline
0.2441& 2.191504&-0.220113&-0.220397&-0.929022& -2.856629& -2.970847 \\
0.1220&2.199625&-0.214254&-0.214327& -0.913327& -2.791699&-2.819631 \\
0.0610&2.201774&-0.212756&-0.212774&-0.909377&-2.774697&-2.781648\\
0.0305&2.202323&-0.212378&-0.212383&-0.908387& -2.770389&-2.772125\\
0.0153&2.202461&-0.212284&-0.212285&-0.908139&-2.769308&-2.769742\\
\hline
\end{tabular}
\end{center}
\caption{The largest six Steklov eigenvalues for the square $n(x)=4$.}
\label{EigS1}
\end{table}
\end{center}

\begin{center}
\begin{table}[h!]
\begin{center}
\begin{tabular}{lrrrrrr}
\hline
 $h$& 1st&2nd&3rd&4th&5th&6th\\
\hline \hline
0.2383&2.507719&0.840066&0.117326&-1.103880&-1.112608&-1.464999\\
0.1192&2.526360&0.851538&0.122637&-1.090066&-1.096900& -1.429175\\
0.0596&2.531439&0.855499&0.124041&-1.086500&-1.092730&-1.420001\\
0.0298&2.532762&0.856926&0.124402&-1.085600& -1.091620&-1.417681\\
0.0149&2.533099&0.857457&0.124494&-1.085374&-1.091319& -1.417098\\
\hline
\end{tabular}
\end{center}
\caption{The largest six Steklov eigenvalues for the L-shaped domain $n(x)=4$.}
\label{EigL1}
\end{table}
\end{center}

In Fig.~\ref{ConvRateReal}, we show the convergence rates of Steklov eigenvalues of three domains.
Since we use the linear Lagrange finite element, the second order convergence is achieved for the unit circle and the square.
For the L-shaped domain, which is non-convex, the convergence rate of the second Steklov eigenvalue is lower than $2$, while the other $5$ 
eigenvalues have second order convergence.
\begin{figure}[h!]
\begin{center}
\begin{tabular}{cc}
\resizebox{0.5\textwidth}{!}{\includegraphics{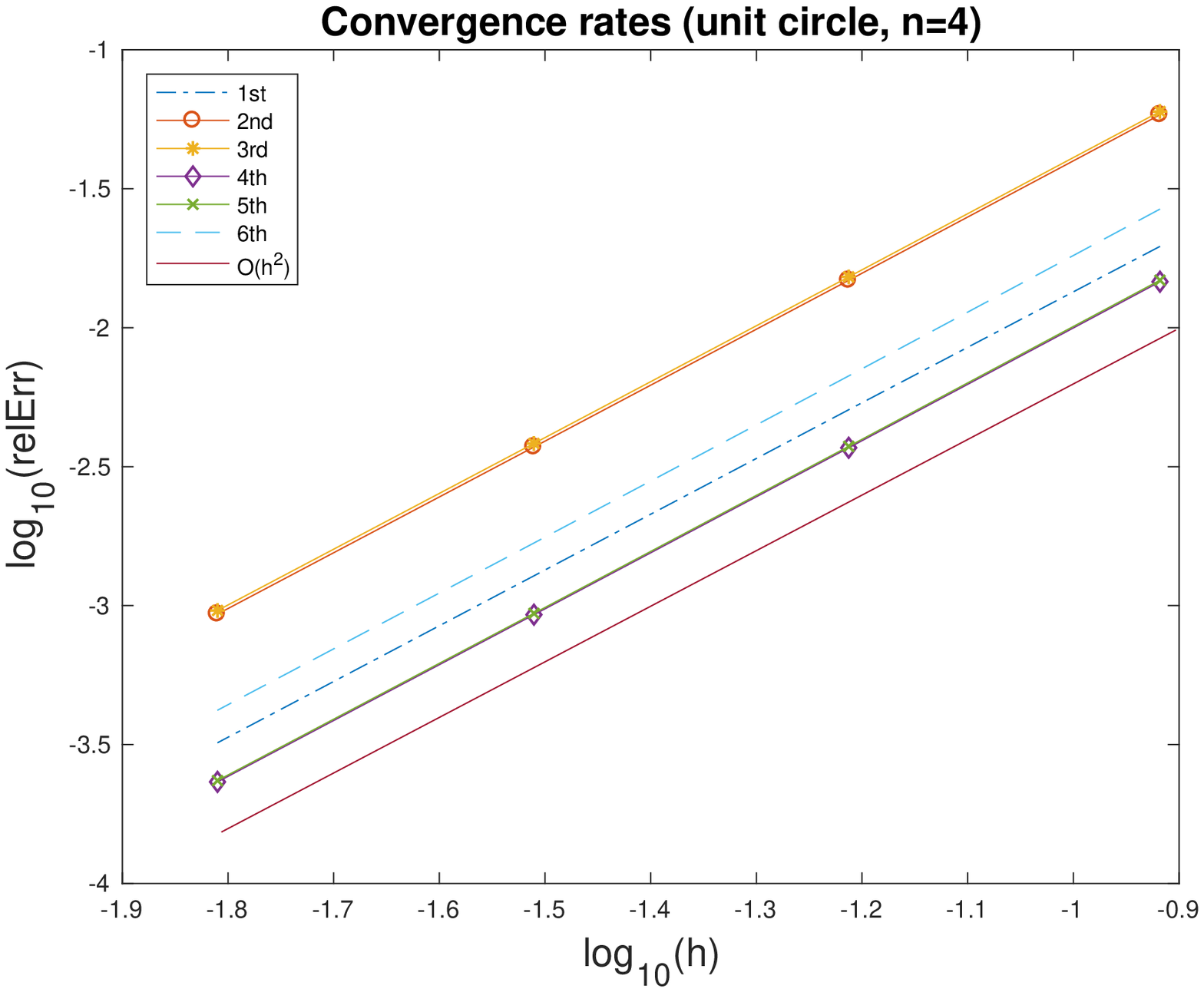}}&
\resizebox{0.5\textwidth}{!}{\includegraphics{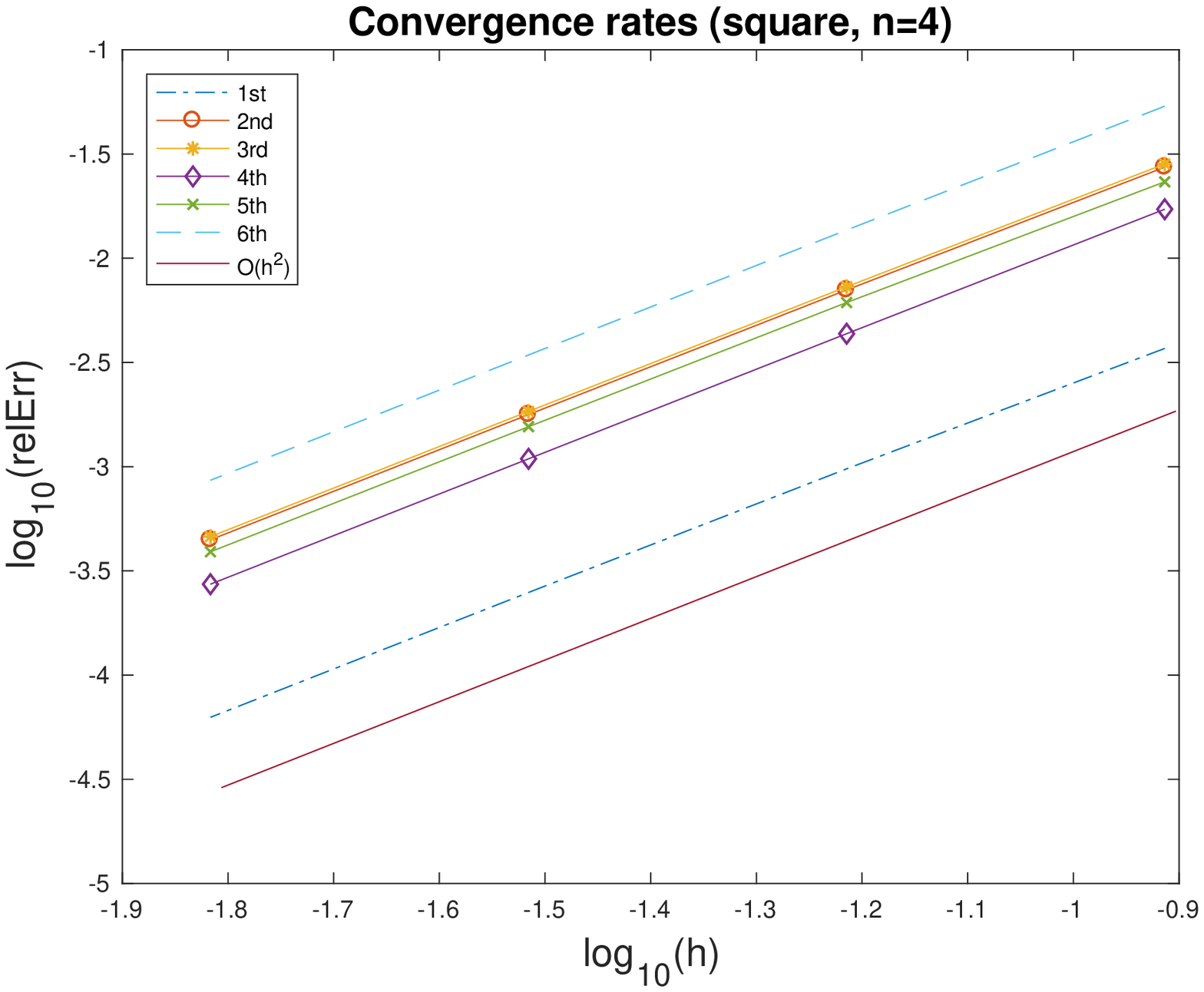}}\\
\resizebox{0.5\textwidth}{!}{\includegraphics{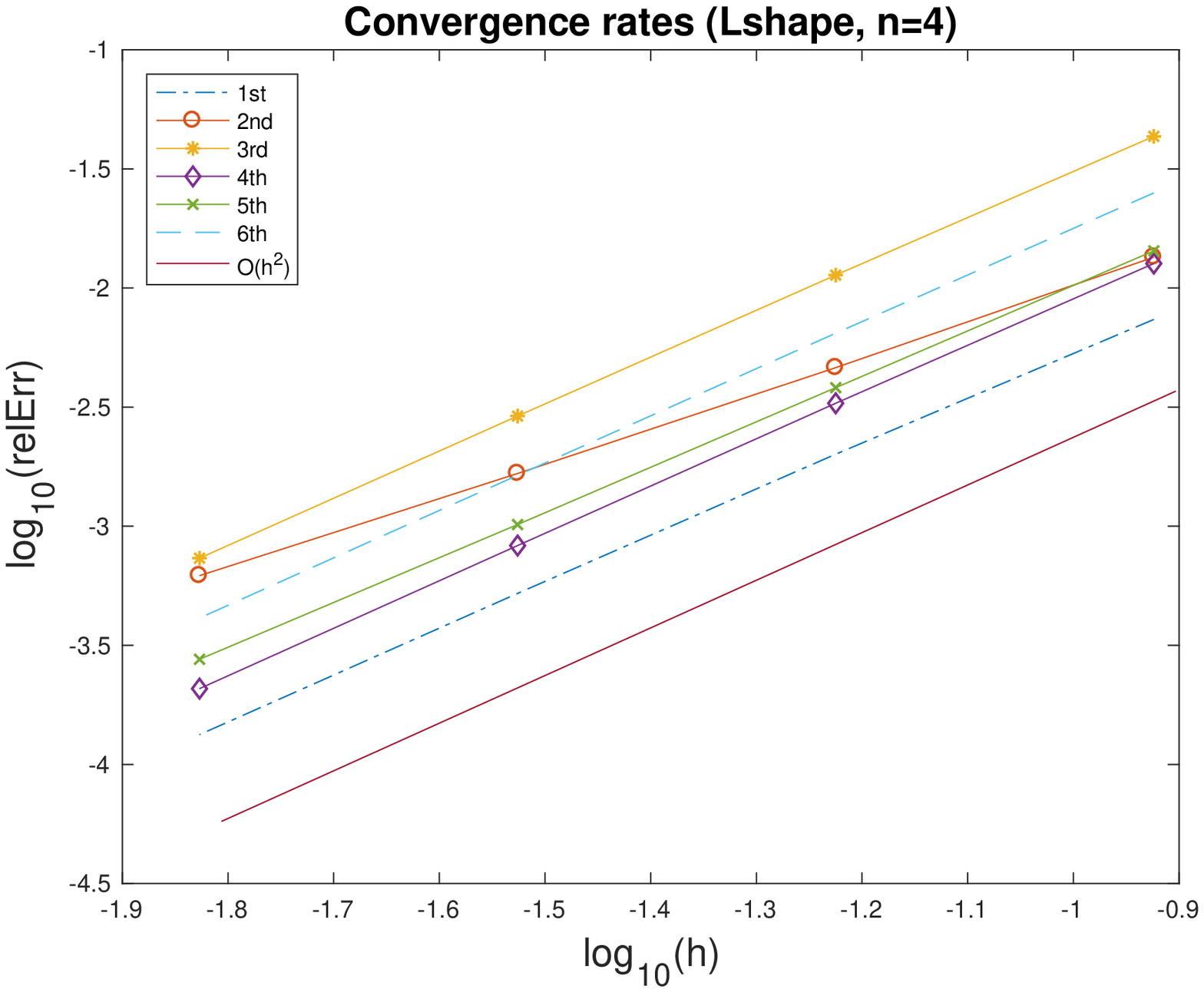}}
\end{tabular}
\end{center}
\caption{Convergence rates of Steklov eigenvalues $n(x)=4$. Top Left: the unit circle. Top Right: the square. Bottom Left: the L-shaped domain.}
\label{ConvRateReal}
\end{figure}

Note that the convergence rate of the eigenvalues relates to the regularity of the associated eigenfunctions. The result in Fig.~\ref{ConvRateReal} implies that
the eigenfunction associated with the second eigenvalues does not belong to $H^2(\Omega)$.

\subsection{Non-selfadjoint Cases} 
When $n(x)$ is complex, the solution operator is non-selfadjoint. 
Consequently, we end up with non-Hermitian matrix eigenvalue problem. Computation of complex eigenvalues of non-Hermitian matrices
are challenging, in particular, when there is no a priori spectral information on the number and distribution of eigenvalues.
To this end, we use the new spectral indicator method introduced in the previous section to compute Steklov eigenvalues.
For all examples, we take $n(x) = 4+4i$.

The left picture of Fig.~\ref{RIMregion} shows the distribution of Steklov eigenvalues for the unit disk on the complex plane.
These are eigenvalues computed using Matlab `eig' for the non-Hermitian matrix ($2097 \times 2097$) resulting from the finite element method. 
The mesh size is $h \approx 0.0613$. Note that `eig' is a direct eigensolver and not suitable for larger matrices. 
\begin{figure}
\begin{center}
\begin{tabular}{cc}
\resizebox{0.5\textwidth}{!}{\includegraphics{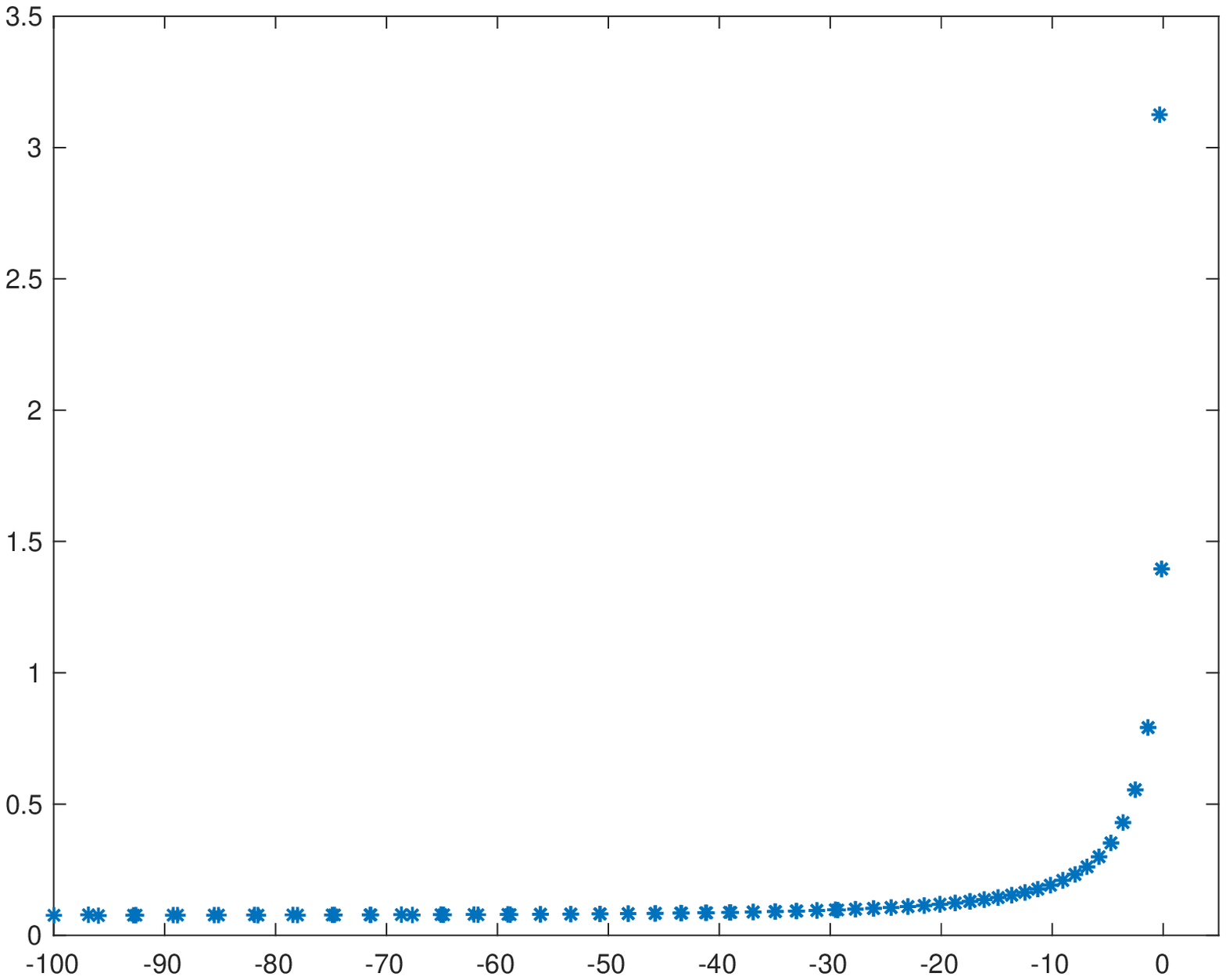}}&
\resizebox{0.5\textwidth}{!}{\includegraphics{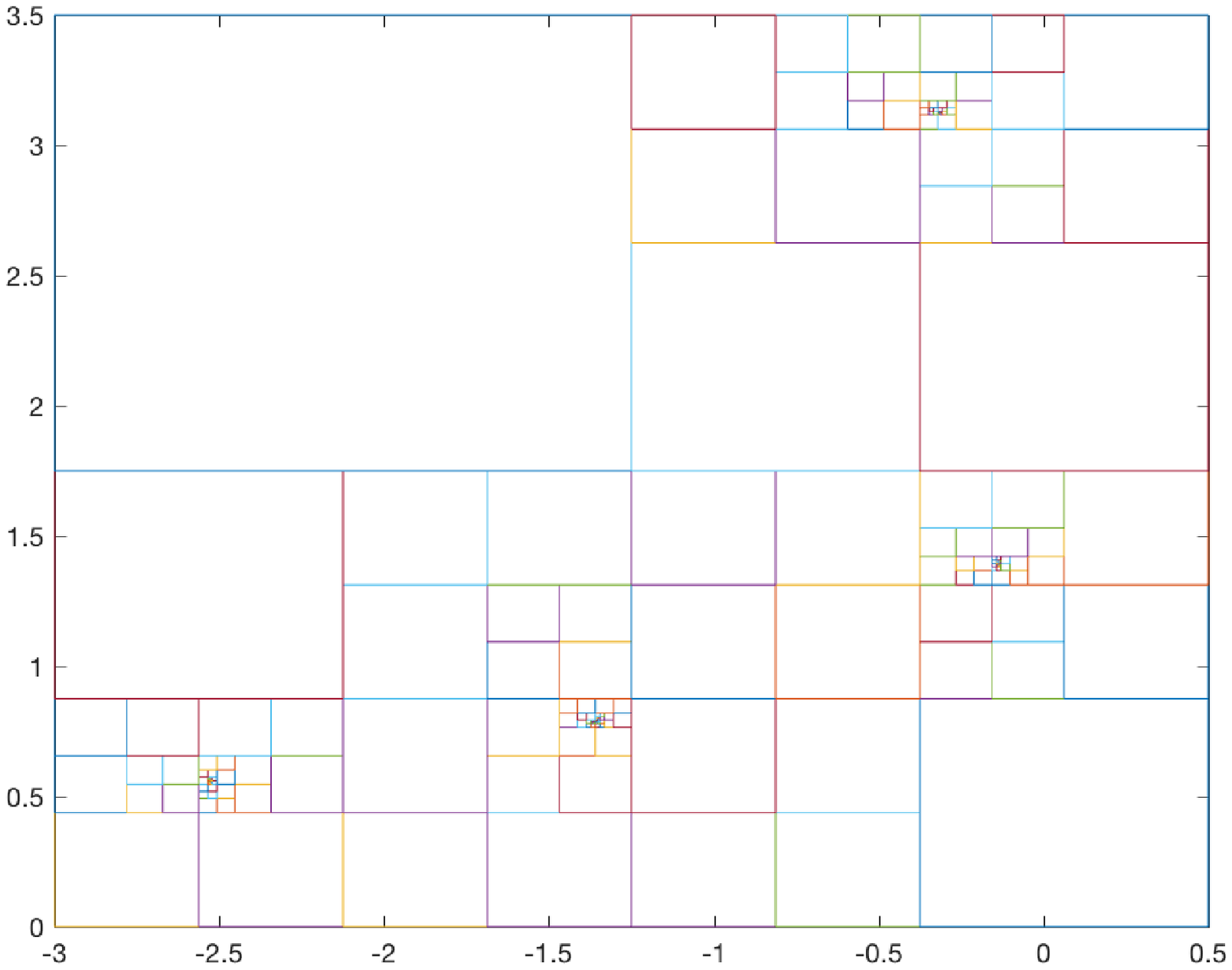}}
\end{tabular}
\end{center}
\caption{ Left: Steklov eigenvalues for the unit disk on the complex plane $n(x)=4+4i$. Right: the region explored by {\bf S-RIM}.}
\label{RIMregion}
\end{figure}
Since we are interested in eigenvalues close to the origin, we choose the search region $S$ on the complex plane to be the square $[-3, 0.5] \times [0, 3.5]$.
The right picture of Fig.~\ref{RIMregion} shows how {\bf S-RIM} explores $S$ and finds the eigenvalues inside $S$.

The computed complex eigenvalues for the three domains are shown in Tables \ref{EigCR1},  \ref{EigCS1}, and \ref{EigCL1}, respectively. 
We arrange the eigenvalues according to the decreasing order of their imaginary parts. Again, these values are consistent with the values given
in \cite{Cakoni2016SIAMAM}, which are reconstructed by some inverse algorithm using scattering data. 
\begin{center}
\begin{table}[h!]
\begin{center}
\begin{tabular}{lrrrr}
\hline
 $h$& 1st&2nd&3rd&4th\\
\hline \hline
0.2341& $\begin{array}{l} -0.298121 \\  \quad+ 3.131620i \end{array}$&
$\begin{array}{l} -0.134181\\ \quad + 1.375387i \end{array}$&
$\begin{array}{l} -0.133990 \\ \quad + 1.374565i \end{array}$&
$\begin{array}{l} -1.371155\\ \quad + 0.786327i \end{array}$ \\ \hline
0.1208& $\begin{array}{l} -0.314981 \\ \quad + 3.126494i \end{array}$&
$\begin{array}{l}  -0.136106 \\ \quad + 1.391267i \end{array}$&
$\begin{array}{l} -0.136049 \\ \quad + 1.391044i \end{array}$&
$\begin{array}{l}  -1.357526\\ \quad+ 0.790318i \end{array}$\\ \hline
0.0613& $\begin{array}{l} -0.319127 \\ \quad + 3.125146i\end{array}$&
$\begin{array}{l}  -0.136650\\ \quad + 1.395302i \end{array}$&
$\begin{array}{l}  -0.136666 \\ \quad + 1.395359i \end{array}$&
$\begin{array}{l}  -1.354126 \\ \quad + 0.79135i \end{array}$\\ \hline
0.0310&$\begin{array}{l}  -0.320161 \\ \quad + 3.124804i \end{array}$&
$\begin{array}{l}  -0.136812 \\ \quad + 1.396392i\end{array}$&
$\begin{array}{l}  -0.136808 \\ \quad + 1.396378i \end{array}$&
$\begin{array}{l}  -1.353338\\ \quad + 0.791628i \end{array}$\\ \hline
0.0155& $\begin{array}{l} -0.320420 \\ \quad + 3.124718i \end{array}$&
$\begin{array}{l}  -0.136849 \\ \quad + 1.396651i \end{array}$&
$\begin{array}{l}  -0.136848 \\ \quad + 1.396647i \end{array}$&
$\begin{array}{l}  -1.353145\\ \quad + 0.791701i \end{array}$\\
\hline
\end{tabular}
\end{center}
\caption{Eigenvalues for the circle $n(x)=4+4i$.}
\label{EigCR1}
\end{table}
\end{center}

\begin{center}
\begin{table}[h!]
\begin{center}
\begin{tabular}{lrrrr}
\hline
 $h$& 1st&2nd&3rd&4th\\
\hline \hline
0.2441&$\begin{array}{l} 0.698699\\ \quad + 2.495471i \end{array}$&
 $\begin{array}{l} -0.344215\\ \quad + 0.843688i \end{array}$&
 $\begin{array}{l} -0.344302\\ \quad + 0.843436i \end{array}$&
 $\begin{array}{l} -0.968470\\ \quad + 0.538448i \end{array}$\\  \hline
0.1220&$\begin{array}{l} 0.689736 \\ \quad + 2.495375i \end{array}$&
 $\begin{array}{l} -0.343337 \\ \quad + 0.848883i \end{array}$&
 $\begin{array}{l} -0.343320\\ \quad + 0.848942i \end{array}$&
 $\begin{array}{l} -0.954693\\ \quad + 0.539666i \end{array}$\\  \hline
0.0610&$\begin{array}{l} 0.687363\\ \quad + 2.495317i \end{array}$&
 $\begin{array}{l} -0.343117\\ \quad + 0.850277i \end{array}$&
 $\begin{array}{l} -0.343114\\ \quad + 0.850292i \end{array}$&
 $\begin{array}{l} -0.951256\\ \quad + 0.539987i \end{array}$\\  \hline
0.0305&$\begin{array}{l} 0.686756 \\ \quad + 2.495300i \end{array}$&
 $\begin{array}{l} -0.343064\\ \quad + 0.850629i \end{array}$&
 $\begin{array}{l} -0.343063\\ \quad + 0.850632i \end{array}$&
  $\begin{array}{l} -0.950397 \\ \quad+ 0.540069i \end{array}$\\  \hline
 0.0151&$\begin{array}{l} 0.686603\\ \quad + 2.495295i \end{array}$&
 $\begin{array}{l} -0.343051 \\ \quad+ 0.850718i \end{array}$&
 $\begin{array}{l} -0.343051 \\ \quad+ 0.850717i \end{array}$&
  $\begin{array}{l} -0.950182 \\ \quad+ 0.540090i \end{array}$\\
\hline
\end{tabular}
\end{center}
\caption{Eigenvalues for the square $n(x)=4+4i$.}
\label{EigCS1}
\end{table}
\end{center}

\begin{center}
\begin{table}[h!]
\begin{center}
\begin{tabular}{lrrrr}
\hline
 $h$& 1st&2nd&3rd&4th\\
\hline \hline
0.2383&$\begin{array}{l} 0.548195\\ \quad + 2.892865i\end{array}$&
  $\begin{array}{l} 0.392629 \\ \quad+ 1.445484i\end{array}$&
 $\begin{array}{l} -0.077110 \\ \quad+ 1.035407i\end{array}$&
 $\begin{array}{l} -1.157394 \\ \quad+ 0.529887i\end{array}$\\ \hline
0.1192&$\begin{array}{l} 0.523142 \\ \quad+ 2.885329i\end{array}$&
 $\begin{array}{l}  0.394633 \\ \quad+ 1.454461i\end{array}$&
 $\begin{array}{l} -0.077154 \\ \quad+ 1.040772i\end{array}$&
 $\begin{array}{l} -1.146157 \\ \quad+ 0.529839i\end{array}$\\ \hline
0.0596&$\begin{array}{l} 0.516546\\ \quad + 2.883129i\end{array}$&
  $\begin{array}{l} 0.395906 \\ \quad+ 1.457375i\end{array}$&
 $\begin{array}{l} -0.07717 \\ \quad+ 1.042191i\end{array}$&
 $\begin{array}{l} -1.143291\\ \quad + 0.529817i\end{array}$\\ \hline
0.0298&$\begin{array}{l} 0.514857\\ \quad + 2.882533i\end{array}$&
  $\begin{array}{l} 0.396543 \\ \quad+ 1.458387i\end{array}$&
 $\begin{array}{l} -0.077178 \\ \quad+ 1.042555i\end{array}$&
 $\begin{array}{l} -1.142571 \\ \quad+ 0.529812i\end{array}$\\ \hline
0.0149&$\begin{array}{l} 0.514430 \\ \quad+ 2.882377i\end{array}$&
  $\begin{array}{l} 0.396829 \\ \quad+ 1.458755i\end{array}$&
 $\begin{array}{l} -0.077177 \\ \quad+ 1.042647i\end{array}$&
 $\begin{array}{l} -1.142391 \\ \quad+ 0.529811i\end{array}$\\ 
\hline
\end{tabular}
\end{center}
\caption{Eigenvalues for the L-shaped domain $n(x)=4+4i$.}
\label{EigCL1}
\end{table}
\end{center}

In Fig.~\ref{ConvRateComplex}, we show the convergence rates of complex Steklov eigenvalues. The second order convergence is achieved for the unit circle
and square. Similar to real $n(x)$, the second eigenvalue of the L-shaped domain has lower convergence rate indicating that the associated eigenfunction has lower regularity.
\begin{figure}
\begin{center}
\begin{tabular}{cc}
\resizebox{0.5\textwidth}{!}{\includegraphics{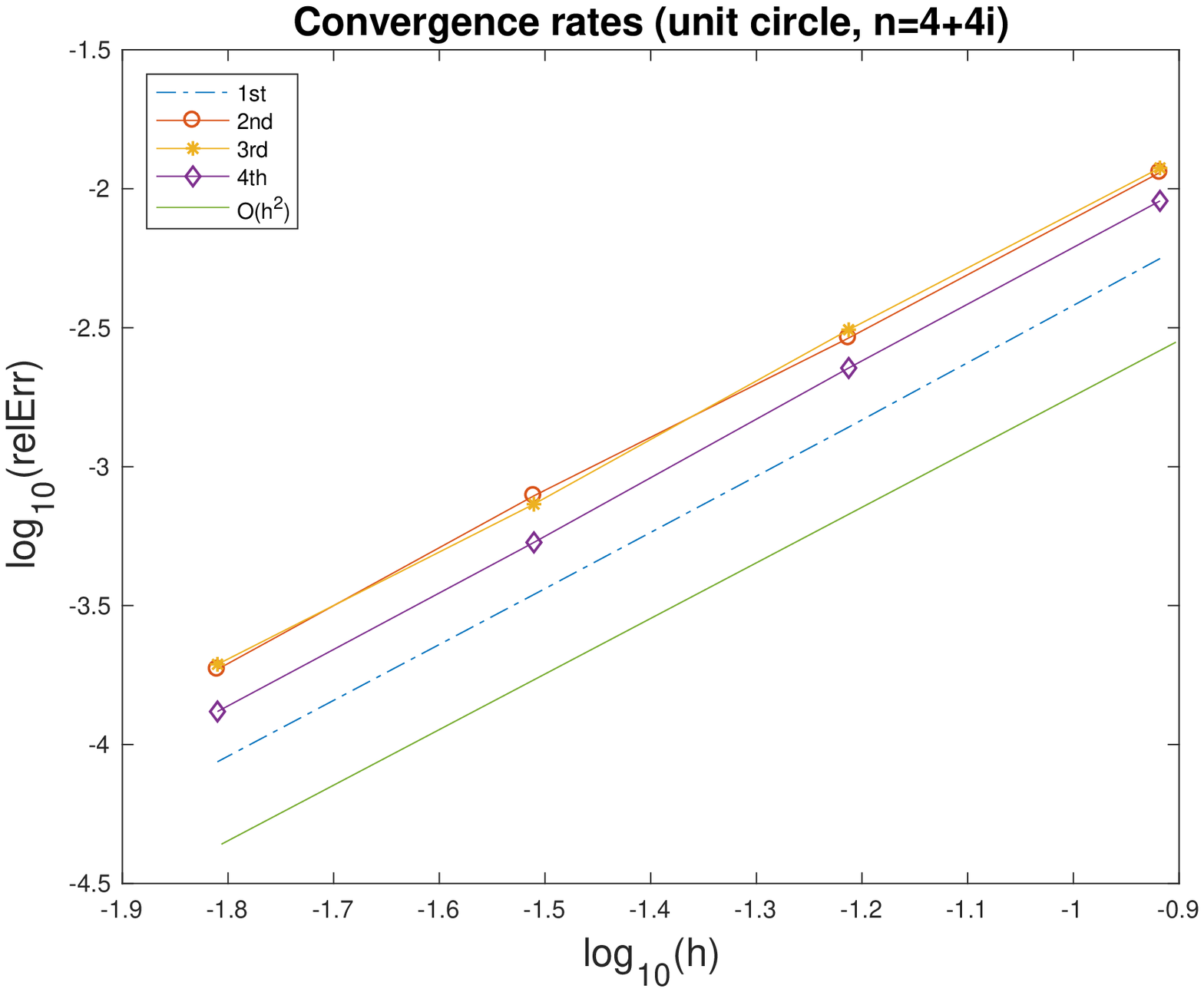}}&
\resizebox{0.5\textwidth}{!}{\includegraphics{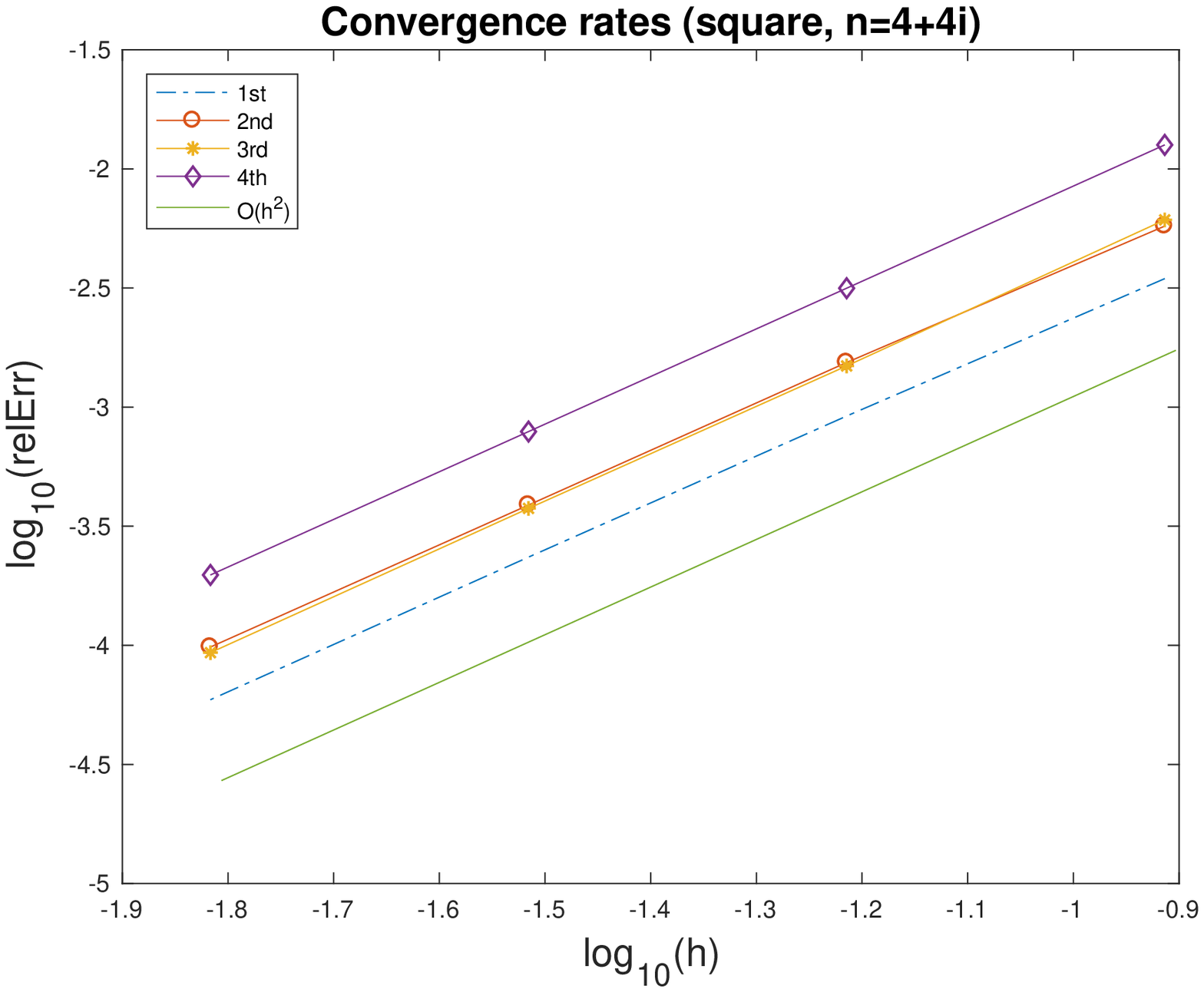}}\\
\resizebox{0.5\textwidth}{!}{\includegraphics{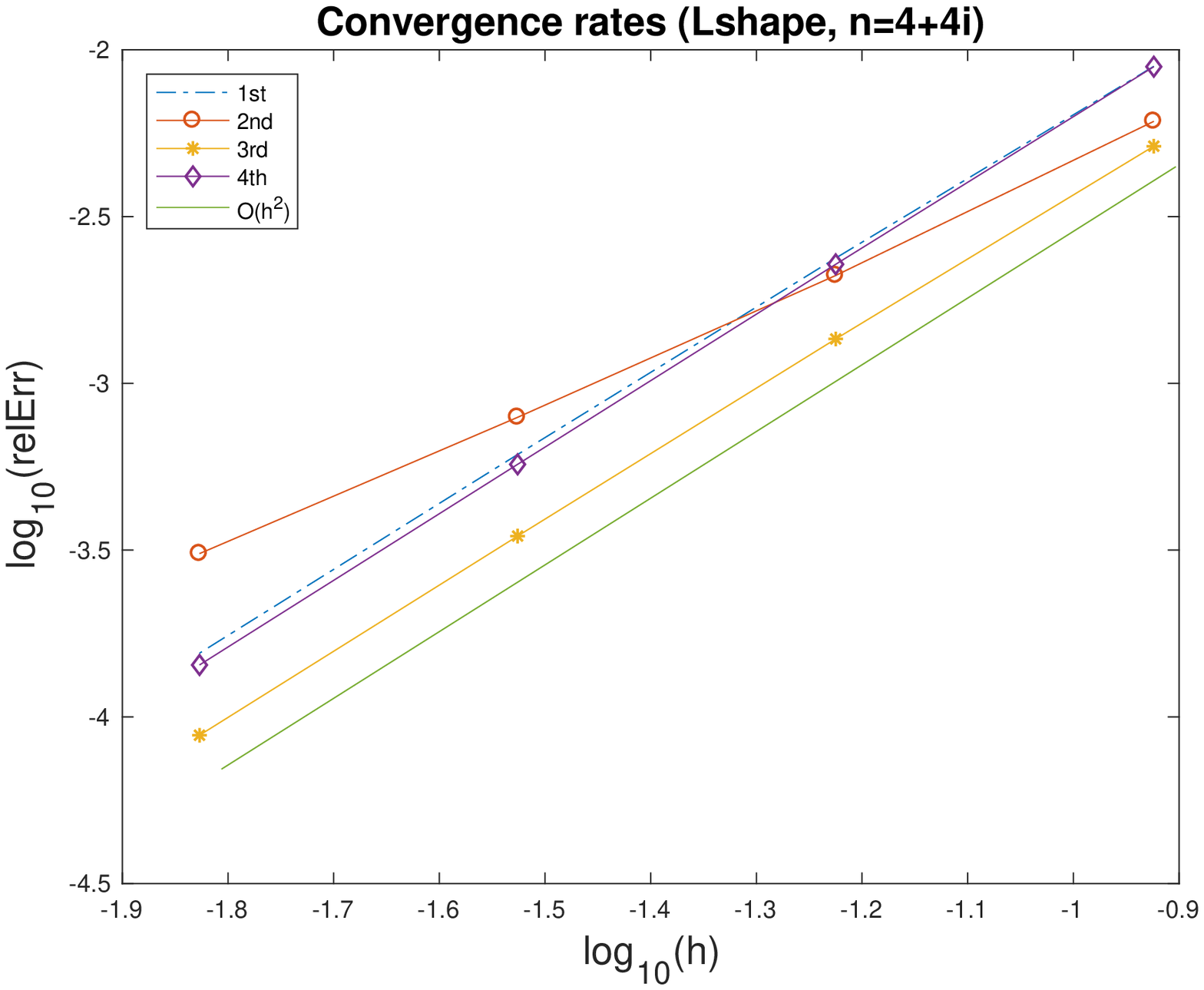}}
\end{tabular}
\end{center}
\caption{Convergence rates of Steklov eigenvalues ($n(x)=4+4i$). Top Left: the unit circle. Top Right: the square. Bottom Left: the L-shaped domain.}
\label{ConvRateComplex}
\end{figure}

\section{Conclusions and Future Works}
In this paper, we study the computation of a non-selfadjoint Steklov eigenvalue problem arising from the inverse scattering theory.
To the authors' knowledge, this is the first numerical paper containing both theory and numerical examples.
An early paper by Bramble and Osborn considered a similar but different non-selfadjoint Steklov eigenvalue problem \cite{BrambleOsborn1972}.
The second order non-selfadjoint operator is assumed to be uniformly elliptic and no numerical results were reported therein.

The contribution of the paper is as follows. The convergence of Lagrange finite elements is proved using the spectral perturbation theory for compact operators. 
Due to the fact that the problem is non-selfadjoint and no a priori spectral information is available, the recently developed spectral indicator method {\bf RIM}
is considered for the resulting non-Hermitian matrix eigenvalue problems.
To improve efficiency, we derive an equivalent but much smaller matrix eigenvalue problem involving only boundary unknowns.
Then a modified version of {\bf RIM} is developed to compute (complex) eigenvalues.

Non-selfadjoint Steklov eigenvalue problems have many important applications. Numerical computation of these problems is challenging.
The problem considered in this paper is related to the Helmholtz equation. Similar problems exist for the Maxwell equation and elasticity equation.
In future, we plan to extend the theory and algorithm here to these problems.




\end{document}